\documentclass[12pt]{article}
\usepackage{amsfonts}
\usepackage{mathrsfs}
\usepackage{amsthm}
 \usepackage{graphicx} 
\usepackage{amsmath} 
\usepackage{amssymb} 
\usepackage{enumerate}
\usepackage{epsfig}
\usepackage{multirow}
\usepackage{color}
\usepackage{cite}

\usepackage{epsfig}

\newtheorem{theorem}{\hskip\parindent Theorem}
\newtheorem{lemma}{\hskip\parindent Lemma}

\newtheorem{remark}{\hskip\parindent Remark}
\numberwithin{equation}{section}

\def\Re {\mathop{\rm Re}\nolimits}

\begin{document}


\title{A note on the connection problem of some special Painlev\'{e} V functions}

\author{Wen-Gao Long$^{a}$, Zhao-Yun Zeng$^{b,}$\footnote{Corresponding author. Email:jenseng5@163.com.
}\ \ and Jian-Rong Zhou$^{c}$
}

\date{\small $^{a}$School of Mathematics and Computational Science, Sun Yat-Sen University, Guangzhou 510275, PR China\\
$^{b}$School of Mathematics and Physics, Jinggangshan University, Ji'an 343009, PR China\\
$^{c}$Department of Mathematics, Foshan University, Foshan 528000,  PR China
}

\maketitle


\begin{abstract}
As a new application of the method of \lq\lq{uniform asymptotics}\rq\rq\  proposed by Bassom, Clarkson, Law and McLeod, we provide a simpler and more rigorous proof of the connection formulas of some special solutions of the fifth Painlev\'{e} equation, which have been established earlier by Andreev and Kitaev.
\end{abstract}

\vspace{5mm}




%

\section{Introduction and main results}
In this paper, as an application of
the method of  \lq\lq{uniform asymptotics}\rq\rq\  introduced by Bassom, Clarkson, Law and
McLeod \cite{APC}, we study the connection formulas of the following Painlev\'{e} V equation
\begin{equation}\label{special PV equation}
\frac{d^2y}{dt^2}=\left(\frac{1}{2y}+\frac{1}{y-1}\right)\left(\frac{dy}{dt}\right)^2-
\frac{1}{t}\frac{dy}{dt}+\frac{y}{t}-\frac{y(y+1)}{2(y-1)},
\end{equation}
which is a special form of the general PV equation (cf. \cite[(5.4.9)]{FAS}) with the parameters $\Theta_0=\Theta_1=\Theta_{\infty}=0$.
The special PV equation (\ref{special PV equation}) has important applications in differential geometry of surfaces. For example, if we set
$y(t)=\left(\frac{e^{i q(x)}+1}{e^{i q(x)}-1}\right)^{2}$ and $t=4x$
in (\ref{special PV equation}), then $q(x)$ satisfies the following equation
\begin{equation*}
xq''-2x{\rm sin}2q+q'+2{\rm sin}q=0,
\end{equation*}
which was considered in \cite{BK} and its solutions are connected with the
problem of classification for rotation surfaces with harmonic inverse mean curvature.


Before stating our main results, we first recall some of the
relevant facts from the isomonodromy formalism for the fifth Painlev\'{e} transcendent
presented in \cite{FAS,MJ}. The Lax pair of the fifth Painlev\'{e} equation, with the special parameters $\Theta_0=\Theta_1=\Theta_{\infty}=0$, is a system of linear ordinary differential equations for the matrix function $Y(\lambda,t)$
\begin{eqnarray}
\frac{dY}{d\lambda}=\left[\frac{t}{2}\sigma_3+\frac{1}{\lambda}
\left(\begin{array}{cc}
v&~~-uv\\\frac{v}{u}&~~-v\end{array}\right)+\frac{1}{\lambda-1}
\left(\begin{array}{cc}
-v&~~uyv\\-\frac{v}{uy}&~~v\end{array}\right)\right]Y
\label{1.2}
\end{eqnarray}
and
\begin{eqnarray*}
\frac{dY}{dt}=\left(\begin{array}{cc}
\frac{1}{2}&~~\frac{u}{\lambda}v(1-y)\\[0.2cm]
\frac{1}{u\lambda}v(1-\frac{1}{y})&~~-\frac{1}{2}\end{array}\right)Y,
\end{eqnarray*}
where $\sigma_3=\left(\begin{array}{cc}1&~0\\0&~-1\end{array}\right)$ and $y(t),v(t)$ and $u(t)$ satisfy the following system of equations
\begin{eqnarray}\label{1.4}
 \left\{\begin{array}{ll}t\frac{dy}{dt}=ty-2v(y-1)^2,\\[0.2cm]
t\frac{dv}{dt}=yv^2-\frac{1}{y}v^2,\\[0.2cm]
t\frac{d}{dt}{\rm ln}u=-2v+yv+\frac{1}{y}v.
 \end{array}
 \right.
 \end{eqnarray}
Furthermore, $y(t)$ is
the solution of the special fifth Painlev\'{e} equation (\ref{special PV equation}).

In a neighborhood of the irregular singular point $\lambda=\infty$, the canonical solutions $Y^{(k)}(\lambda)$
have the following asymptotic expansion
 \begin{eqnarray}\label{1.8}
 Y^{(k)}(\lambda)=\left(I+\left(\begin{matrix}v-\frac{v^{2}(1-y)^{2}}{ty}&\frac{uv(1-y)}{t}\\\frac{v(y-1)}{uyt}&-v+\frac{v^{2}(1-y)^{2}}{ty}\end{matrix}\right)\frac{1}{\lambda}+\mathcal{O}\left(\frac{1}{\lambda^2}\right)\right){\rm exp}\left(\frac{\lambda t}{2}\sigma_3\right), t\in \mathbb{R}_+
 \end{eqnarray}
 as $|\lambda|\rightarrow\infty$ in the corresponding Stokes sectors
 $$\Omega^{(k)}:=\left\{\lambda\in\mathbb{C},~-\frac{\pi}{2}+\pi(k-2)<{\rm arg}~\lambda<\frac{3\pi}{2}+\pi(k-2)\right\}, ~k=1, 2.$$
These canonical solutions are related by the Stokes matrices $S_k$,
 \begin{eqnarray}
 Y^{(k+1)}(\lambda)=Y^{(k)}(\lambda)S_k,~~~\lambda\in \Omega^{(k)}\cap \Omega^{(k+1)}.
 \label{1.9}
 \end{eqnarray}
 Furthermore, the Stokes matrices $S_k$ can be written as
 \begin{eqnarray}\label{Stokes matrices}
 S_{2k+1}=\left(\begin{array}{cc}
1&~0\\s_{2k+1}&~1\end{array}\right),~~S_{2k}=\left(\begin{array}{cc}
1&~s_{2k}\\0&~1\end{array}\right)
\end{eqnarray}
where $s_k(k=1, 2)$ are called the Stokes multipliers. If we now differentiate
 (\ref{1.9}) with respect to $t$ and use the fact that $Y^{(k)}$ and $Y^{(k+1)}$ satisfy
  (\ref{1.2}), we immediately obtain that $S_k$ is independent of $t$.  This is the isomonodromy condition.

There exists a unique solution (cf.\cite[Sec.3]{FA3}) of system (\ref{1.4}) with the following asymptotic behaviors:
 \begin{eqnarray}\label{behavior-yuv-infinity}
  y=-1-\frac{4}{t}+O(t^{-2}),~~v=-\frac{t}{8}+O(t^{-1}),~~u=\hat{u}e^{\frac{t}{2}}(1+o(1))
  \label{1.5}
  \end{eqnarray}
as $t \rightarrow+\infty$, and
\begin{equation}\label{behavior-yuv-0}
 \begin{split}
y&=\frac{(\sigma s^2t^{\sigma}-2)^2}{(\sigma s^2t^{\sigma}+2)^2}+O(t),~~v=\frac{1}{4s^2t^{\sigma}}-\frac{\sigma^{2} s^2t^{\sigma}}{16}+O(t),
\\
u&=-r\frac{2+\sigma s^2t^{\sigma}}{2-\sigma s^2t^{\sigma}}(1+o(1))
 \end{split}
\end{equation}
as $t \rightarrow 0^+$, where $\hat{u},\sigma$ and $r$ are complex constants, and
 \begin{eqnarray}
 s^2=\frac{i\sigma^2}{4\pi^3}\frac{\Gamma^2(-\sigma)}{\Gamma^2(\sigma)}\Gamma^6(\frac{\sigma}{2}).
\label{connection formula2}
\end{eqnarray}

The above connection formulaes (\ref{behavior-yuv-infinity})-(\ref{connection formula2}) have been
established by Andreev and Kitaev in \cite{FA3}.  These results were derived by proposed a certain limit procedure of the parameters $\Theta_j~(j=0,1,\infty)$ based on the results by using
isomonodromy deformation and the WKB method in their previous work \cite{FA4}. It is worth noting that the limit procedure of  $\Theta_j~(j=0,1,\infty)$ tend to zero in \cite{FA3} seems neither obvious nor easy-to-prove.

In this paper, to avoid the limit process, we  consider the equation (\ref{special PV equation}) directly, i.e. the general PV equation with $\Theta_j=0~(j=0,1,\infty)$. By careful analysis, we find that (\ref{behavior-yuv-infinity}) can be immediately derived from (\ref{1.4}), while (\ref{behavior-yuv-0}) may depend on the monodromy matrices for the regular singular points $\lambda=0$ and $1$, which have been studied by Jimbo \cite{MJ} via the method of isomonodromy deformation. Yet the main objective of the present paper is to justify connection
formulas between parameters involved in the asymptotic approximations (\ref{behavior-yuv-infinity}) and  (\ref{behavior-yuv-0}) via the Stokes multipliers for the irregular singular point $\lambda=\infty$ by using of the method of ``uniform asymptotics''. Although the asymptotics behaviors in (\ref{behavior-yuv-0}) are not available by virtue of the method of ``uniform asymptotics'', fortunately we can obtain the following asymptotics behaviors
\begin{equation}\label{behavior-yuv-0+}
\frac{v^2 (y-1)^2}{y}=\frac{\sigma^2}{4}+o(1),\quad uy^{\frac{1}{2}}=-r+o(1), \quad \text{as } t\rightarrow 0^{+},
\end{equation}
provided that $v=o(t^{-1})$ as $t\rightarrow 0^{+}$, where $\sigma$ and $r$ are complex constants.
 It is readily observed that (\ref{behavior-yuv-0+}) is consistent with (\ref{behavior-yuv-0}). Hence, our main task in this paper is to establish the relations between the parameter $\hat{u}$ in (1.4) and the parameters $\sigma, r$ in (\ref{behavior-yuv-0+}). The results can be stated as the following theorem. 
\begin{theorem}
{\rm (cf. \cite[Cor.3.2]{FA3})} If $ 0\leq \Re\sigma<1$, then the relation between the parameter $\hat{u}$ in (\ref{1.5}) and the parameters $\sigma, r$ in (\ref{behavior-yuv-0+}) are given by
\begin{eqnarray}
  \sigma=\frac{i}{\pi}\ln(3+\sqrt{8}),~r=-i\hat{u}.
\label{connection formula}
\end{eqnarray}
\end{theorem}

 In this paper, we shall provide a hopefully simpler and more rigorous derivation of the connection formulas in \eqref{connection formula} via the method of  \lq\lq{uniform asymptotics}\rq\rq,  which was first proposed by Bassom et al. in \cite{APC} and further developed in \cite{Lu-Shao-2002,WZ1,WZ2,ZZ2015}. The difference between this approach and the WKB method is that
 the latter needs a matching process (cf. \cite[Sec.7]{FA4}) in different Stokes domain, while the former does not need such a complicated procedure. Therefore, to some extent, the method of ``uniform asymptotics" is an improvement over WKB.

Although the basic ideas of our approach are taken from \cite{APC} for PII, there are some technical differences between the cases of PII and PV. For instance, in the case of PII, the second-order ordinary differential equation(ODE) obtained from the Lax pair has only coalescing turning points; see \cite{APC}. Thus, uniform asymptotic approximations of the canonical solutions can be constructed in terms of the parabolic cylinder functions according to \cite{FO}. Recently, the method of \lq\lq{uniform asymptotics}\rq\rq has been applied to the connection problems for PV, cf. Zeng and Zhao \cite{ZZ2015}. However, all these cases also differ from our present case.
When $t\rightarrow+\infty$, $\eta=0$ (see Section 3) is not only the coalescing turning points but also a double pole of the second-order differential equation, hence, the parabolic cylinder function is not available.  By careful analysis, we find that uniform asymptotic approximations can be successfully constructed by the modified Bessel functions according to the ideas of Dunster \cite{DTM}. To the best of our knowledge, under the framework of the method of \lq\lq{uniform asymptotics}\rq\rq, the modified Bessel functions have never been used in the connection problem of the fifth Painlev\'{e} equations, although the Hankel functions are used in \cite{WZ1} for the third Painlev\'{e} equations.

 The remaining part of this paper is organized as follows. In
Section 2, we derive  uniform approximations
to the solutions of the second-order differential equation obtained from the Lax
pair (\ref{1.2}) as $t\rightarrow0^{+}$ by using of the Whittaker functions on the Stokes curves, and then evaluate the Stokes multipliers (see (\ref{t-0-s1}) and (\ref{t-0-s2})) as $t\rightarrow 0^{+}$.  In the last section, we construct uniform approximations to the solutions of
 the second-order differential equation as $t\rightarrow+\infty$  by virtue of the modified Bessel functions on the Stokes curves.  Based on these approximations, we
evaluate the Stokes multipliers as $t\rightarrow+\infty$.
The proof of Theorem 1 is also provided in that section.

\section{The monodromy data for $t\rightarrow 0^{+}$}\setcounter{equation}{0}
In this section we use the method of \lq\lq{uniform asymptotics}\rq\rq \cite{APC} to obtain the Stokes multipliers $s_{1}$ and $s_{2}$ in \eqref{Stokes matrices} for the case of $t\rightarrow 0^+$.

First, we make the scaling $\eta=\lambda t$, then (\ref{1.2}) becomes
\begin{eqnarray}\label{2.2}
\frac{dY}{d\eta}=\left(\begin{array}{cc}
\frac{1}{2}+\frac{v}{\eta}-\frac{v}{\eta-t}&~-\frac{vu}{\eta}+\frac{vuy}{\eta-t}\\[0.2cm]
\frac{v}{u\eta}-\frac{v}{uy(\eta-t)}&~-\frac{1}{2}-\frac{v}{\eta}+\frac{v}{\eta-t}
\end{array}\right)Y
=\left(\begin{array}{cc}
 A & B \\ C & -A
\end{array}\right)Y.
\end{eqnarray}
Set $\phi=C^{-\frac{1}{2}}Y_2$, where $Y=(Y_1,Y_2)^T$ is a fundamental solution of (\ref{2.2}) and $C=\frac{v}{u\eta}(1-\frac{1}{y})$, then
\begin{equation}\label{2.4}
\begin{split}
\frac{d^2\phi}{d\eta^2}=&\left[A^{2}+BC-A'+A\frac{C'}{C}+\frac{3}{4}\left(\frac{C'}{C}\right)^{2}-\frac{1}{2}\frac{C''}{C}\right]\phi\\
=&\left[\frac{1}{4}+\frac{1}{\eta^2}\left(\frac{v^2(1-y)^2}{y}-vt\right)-\frac{1}{2\eta}-\frac{1}{4\eta^2}+g(\eta,t)\right]\phi= F(\eta,t)\phi.
\end{split}
\end{equation}
with $A'=\frac{dA}{d\eta}$ and $C'=\frac{d C}{d\eta}$.

The form of equation (\ref{2.4}) motivates us to consider the following model equation
\begin{eqnarray}{\label{Whittaker equation}}
\frac{d^2\psi}{d\eta^2}+
\left[\frac{\frac{1}{4}-\frac{\alpha(t)^2}{4}}{\eta^2}+\frac{\frac{1}{2}}{\eta}-\frac{1}{4}\right]\psi=0,
\label{2.5}
\end{eqnarray}
where  $\frac{\alpha(t)^{2}}{4}=\frac{v^{2}(1-y)^{2}}{y}-vt$. 
 Furthermore, we find $g(\eta,t)=\mathcal{O}\left(\frac{vt}{\eta^{3}}\right)$ as $\eta\rightarrow\infty$ provided that $v=o(t^{-1})$. 
Noting that (\ref{2.5}) is the Whittaker equation\cite[p334]{OL} with parameters $\kappa=\frac{1}{2}$, $\mu=\frac{\alpha(t)}{2}$,
and it has two linear independent solutions $M_{\frac{1}{2},\frac{\alpha}{2}}(\eta)$ and $W_{\frac{1}{2},\frac{\alpha}{2}}(\eta)$.
Hence, we have the following lemma.



\begin{lemma}\label{lemma-uniform-approximation-1}
There exist two constants $C_{1}$ and $C_{2}$ such that
\begin{equation}
\phi=\left[C_{1}+o(1)\right]M_{\frac{1}{2},\frac{\alpha}{2}}(\eta)+\left[C_{2}+o(1)\right]W_{\frac{1}{2},\frac{\alpha}{2}}(\eta)
\end{equation}
as $t\rightarrow 0^{+}$ uniformly for $\eta$ on two adjacent Stokes curves of (\ref{2.4}) emanating from one of the turning points and terminating at infinity.
\end{lemma}

\begin{proof}
The proof of this lemma is similar to \cite[Theorem 1 or 2]{APC}. Denoting $\psi_{+}=M_{\frac{1}{2},\frac{\alpha}{2}}(\eta)$ and $\psi_{-}=W_{\frac{1}{2},\frac{\alpha}{2}}(\eta)$.
According to the parametrix variation method for the non-homogeneous ODEs, we only need to show the following approximation (cf. \cite[(3.17)]{APC})
\begin{equation}\label{approximation-Integral-1}
\int_{\eta_{0}}^{\eta}\frac{\psi_{+}(\eta)\psi_{-}(s)-\psi_{-}(\eta)\psi_{+}(s)}{W(\psi_{+},\psi_{-})}g(s,t)\phi(s)ds=o(1)(\psi_{+}(\eta)+\psi_{-}(\eta)),~~t\rightarrow 0^+,
\end{equation}
where the path of integration is taken along the Stokes curves, $\eta_{0}$ is one of the turning points, and $W(\psi_{+},\psi_{-})$ is the Wronskian determinant. Obviously,  $\eta_{0}\sim 1\pm\sqrt{2-\alpha(t)^{2}}$ as $t\rightarrow 0^{+}$.

First, we have (cf. \cite[(13.14.26)]{OL})
$W(\psi_{+},\psi_{-})=\mathcal{O}(1)$.
In addition, according to \cite[(13.19.2)(13.19.3)]{OL}, it is easy to obtain $\psi_{\pm}=\mathcal{O}(\eta^{\frac{1}{2}})$ as $\eta\rightarrow\infty$ on the Stokes curves. Combining these two estimates with $g(\eta,t)=\mathcal{O}(\frac{vt}{\eta^{3}})$ as $\eta\rightarrow\infty$,
we conclude that the integral in the left-side of (\ref{approximation-Integral-1}) is integrable.
Finally, noting the condition that $g(\eta,t)=o(1)$ as $t\rightarrow 0^{+}$ uniformly for all $\eta$ on the path of integration, we easily get (\ref{approximation-Integral-1}).
\end{proof}

According to Lemma \ref{lemma-uniform-approximation-1}, $Y_{2}$, the second line of $Y^{(k)}$, i.e. $(Y_{21},Y_{22})$, can be asymptotically approximated by the linear combinations of $C^{\frac{1}{2}}M_{\frac{1}{2},\frac{\alpha}{2}}(\eta)$ and $C^{\frac{1}{2}}W_{\frac{1}{2},\frac{\alpha}{2}}(\eta)$ when $|\eta|\rightarrow+\infty$ in $\Omega^{(k)}$.

Now, we are in a position to evaluate the Stokes multipliers through $Y^{(k+1)}=Y^{(k)}S_{k}$. Although this part is slightly different from the approach in \cite{APC}, they are coincide essentially. See details in Sec.6 of \cite{APC} or the corresponding sections in \cite{WZ1,WZ2,ZZ2015}. Here we only give the derivation of $s_{1}$. To get $s_{2}$, one can repeat the process except for noting that the uniform asymptotic behaviors of the Whittaker functions in (\ref{uniform-asymptotic-behavior-Whittaker}) should be changed. For convenient we denote $C\sim\frac{v}{u\eta}(1-\frac{1}{y})=\frac{\beta^{2}(t)}{\eta}$.

When $\eta\rightarrow\infty$, according to \cite[(13.19.2) (13.19.3)]{OL} and noting that $\eta=t\lambda$, we get
\begin{equation}{\label{uniform-asymptotic-behavior-Whittaker}}
\begin{cases}
C^{\frac{1}{2}}M_{\frac{1}{2},\frac{\alpha}{2}}(\eta)  \sim c_{1}e^{\frac{t}{2}\lambda}\lambda^{-1}+c_{2}e^{-\frac{t}{2}\lambda},\quad \arg\eta\in(-\frac{3\pi}{2},\frac{\pi}{2}),\\[0.2cm]
C^{\frac{1}{2}}M_{\frac{1}{2},\frac{\alpha}{2}}(\eta)  \sim c_{1}e^{\frac{t}{2}\lambda}\lambda^{-1}+c_{2}e^{\alpha(t)\pi i}e^{-\frac{t}{2}\lambda},\quad \arg\eta\in(-\frac{\pi}{2},\frac{3\pi}{2}),\\[0.2cm]
C^{\frac{1}{2}}W_{\frac{1}{2},\frac{\alpha}{2}}(\eta)  \sim
c_{3}e^{-\frac{t}{2}\lambda}\quad |\arg\eta|<\frac{3\pi}{2},
\end{cases}
\end{equation}
where $c_{1}=\beta(t)\frac{\Gamma(1+\alpha(t))}{t\Gamma(\frac{\alpha(t)}{2})}, c_{2}=\beta(t)\frac{\Gamma(1+\alpha(t))}{\Gamma(1+\frac{\alpha(t)}{2})}e^{-\alpha(t)\pi i/2}$ and $c_{3}=\beta(t)$.
On the other hand, the uniform asymptotic behavior of $Y_{21}$ and $Y_{22}$ can be directly obtained from (\ref{1.8}), and the results are
\begin{equation}{\label{canonical-t-0-omega-k}}
Y_{21}\sim \frac{\beta(t)^{2}}{t}e^{\frac{t}{2}\lambda}\lambda^{-1},\quad Y_{22}\sim e^{-\frac{t}{2}\lambda},
\end{equation}
as $\lambda\rightarrow\infty$ in $\Omega^{(k)}$, $k=1,2$.
Hence, it follows from (\ref{uniform-asymptotic-behavior-Whittaker}) and (\ref{canonical-t-0-omega-k}) that
\begin{equation}{\label{Y--pi/2-t-0}}
(Y_{21},Y_{22})\sim C^{\frac{1}{2}}(M_{\frac{1}{2},\frac{\alpha}{2}}(\eta),W_{\frac{1}{2},\frac{\alpha}{2}}(\eta))\left(\begin{matrix}\frac{\beta(t)^{2}}{tc_{1}}&0\\-\frac{\beta(t)^{2}c_{2}}{tc_{1}c_{3}}&\frac{1}{c_{3}}
\end{matrix}\right),\quad \lambda\rightarrow\infty, \lambda\in \Omega^{(1)}
\end{equation}
and
\begin{equation}{\label{Y-pi/2-t-0}}
(Y_{21},Y_{22})\sim C^{\frac{1}{2}}(M_{\frac{1}{2},\frac{\alpha}{2}}(\eta),W_{\frac{1}{2},\frac{\alpha}{2}}(\eta))\left(\begin{matrix}\frac{\beta(t)^{2}}{tc_{1}}&0\\-\frac{\beta(t)^{2}c_{2}e^{\alpha(t)\pi i}}{tc_{1}c_{3}}&\frac{1}{c_{3}}
\end{matrix}\right),\quad \lambda\rightarrow\infty, \lambda\in \Omega^{(2)}
\end{equation}
Noting that in (\ref{Y--pi/2-t-0}), $(Y_{21},Y_{22})$ represent the second line entries of $Y^{(1)}(\lambda)$, while in (\ref{Y-pi/2-t-0}), $(Y_{21},Y_{22})$ are the second line of $Y^{(2)}(\lambda)$,  by the definition of $S_1$ in \eqref{1.9}, we have
\begin{equation}{\label{s1-alpha-t}}
s_{1}\sim\frac{c_{2}\beta(t)^{2}}{tc_{1}}(1-e^{\alpha(t)\pi i}) =-\frac{4i\beta(t)^{2}}{\alpha(t)}\sin(\frac{\alpha(t)\pi}{2}) \sim \frac{2i}{uy^{\frac{1}{2}}}\sin{\frac{\alpha(t)\pi}{2}}
\end{equation}
as $t\rightarrow 0^{+}$. Similar derivations lead to
\begin{equation*}
s_{2}\sim 2iuy^{\frac{1}{2}}\sin(\frac{\alpha(t)\pi}{2}) \quad \text{ as } t\rightarrow 0^{+}.
\end{equation*}
Since $s_{k}, k=1,2$ are independent of $t$, it follows that $\lim\limits_{t\rightarrow 0^{+}}uy^{\frac{1}{2}}$ and $\lim\limits_{t\rightarrow 0^{+}}\alpha(t)$ are exist. Without loss of generality, we may take
\begin{equation}{\label{alpha-t-uy}}
\alpha(t)=\sigma+o(1), \quad  uy^{\frac{1}{2}}=-r+o(1) \text{ as } t\rightarrow 0^{+}.
\end{equation}
Therefore the two Stokes multipliers are
\begin{equation}\label{t-0-s1}
s_{1}=-\frac{2i}{r}\sin(\frac{\sigma\pi}{2})
\end{equation}
and
\begin{equation}\label{t-0-s2}
s_{2}=-2i r\sin(\frac{\sigma\pi}{2}).
\end{equation}

\begin{remark}
In (\ref{s1-alpha-t}), the branch of $\alpha(t)$ is chosen such that $\alpha(t)=\frac{2v(1-y)}{y^{1/2}}+o(1)$ as $t\rightarrow 0^{+}$. If choosing $\alpha(t)=\frac{2v(y-1)}{y^{1/2}}+o(1)$, we will need to set $\alpha(t)=-\sigma+o(1)$ in (\ref{alpha-t-uy}).
\end{remark}

 \section{The monodromy data for $t\rightarrow +\infty$}\setcounter{equation}{0}

In this section, our goal is to derive the Stokes multipliers $s_{1}$ and $s_{2}$ as $t\rightarrow +\infty$ by applying the method of ``uniform asymptotics"\cite{APC}. 

To derive $s_{2}$, let us first make the following transformation in \eqref{1.2}
\begin{equation}{\label{transform-Y-tilde(Y)-1}}
\tilde{Y}(\lambda)=\left(\begin{matrix}1&0\\-1&1\end{matrix}\right)u^{-\frac{\sigma_{3}}{2}}Y(\lambda).
\end{equation}
As a result, we obtain
\begin{equation}\label{equation-tilde-Y}
\frac{d\tilde{Y}(\lambda)}{d\lambda}=\left(\begin{matrix}\tilde{A}&\tilde{B}\\\tilde{C}&-\tilde{A}\end{matrix}\right)\tilde{Y}(\lambda)=\left(\begin{matrix}\bar{A}+\bar{B}&\bar{B}\\\bar{C}-\bar{B}-2\bar{A}&-(\bar{A}+\bar{B})\end{matrix}\right)\tilde{Y}(\lambda),
\end{equation}
where
$$\bar{A}=\frac{t}{2}+\frac{v}{\lambda}-\frac{v}{\lambda-1},~\bar{B}=-\frac{v}{\lambda}+\frac{vy}{\lambda-1}~\text{ and } \bar{C}=\frac{v}{\lambda}-\frac{v}{y(\lambda-1)}.$$

It is easy to check that the transformation \eqref{transform-Y-tilde(Y)-1} does not change the Stokes matrices. Let $\tilde{Y}(\lambda)=(Y_{1},Y_{2})^{T}$ be a fundamental solution of \eqref{equation-tilde-Y}, and set $\tilde{\phi}=\tilde{C}^{-\frac{1}{2}}Y_{2}$, then eliminating $Y_{1}$ from \eqref{equation-tilde-Y} gives
\begin{equation}\label{shrodinger-equation-tilde-Y-1}
\frac{d^{2}\tilde{\phi}}{d\lambda^2}=\left[\tilde{A}^{2}+\tilde{B}\tilde{C}-\tilde{A}'+\tilde{A}\tilde{C}^{-1}\tilde{C}'+\frac{3}{4}(\tilde{C}^{-1}\tilde{C}')^{2}-\frac{1}{2}\tilde{C}^{-1}\tilde{C}''\right]\tilde{\phi},
\end{equation}
where $\tilde{A}'=\frac{d\tilde{A}}{d\lambda}$ and $\tilde{C}'=\frac{d\tilde{C}}{d\lambda}$. For $t\rightarrow\infty$, substituting the asymptotic behaviors of $y$ and $v$ in \eqref{1.5} into \eqref{shrodinger-equation-tilde-Y-1}, we finally obtain the following second-order equation:
\begin{equation}\label{shrodinger-tilde-Y-2}
\begin{split}
\frac{d^{2}\tilde{\phi}}{d\eta^2}=\left[\frac{t^{2}\eta^{2}}{4(\eta^{2}-\frac{1}{4})}+\frac{3}{4\eta^{2}}+\tilde{g}(\eta,t)\right]\tilde{\phi}=\tilde{F}(\eta,t)\tilde{\phi},
\end{split}
\end{equation}
where $\eta=\lambda-\frac{1}{2}$, and $\tilde{g}(\eta,t)=\mathcal{O}(1)$ uniformly for $\eta$ away from both $0$ and $\frac{1}{2}$. Furthermore,
\begin{equation}\label{well property of tilde-g}
\begin{split}
&\tilde{g}(\eta,t)=\mathcal{O}\left(\frac{1}{\eta}\right), \text{ for }\eta\rightarrow 0;~~~~\tilde{g}(\eta,t)=\mathcal{O}\left(\frac{1}{(\eta-\frac{1}{2})^{2}}\right), \text{ for }\eta\rightarrow \frac{1}{2};\\&\tilde{g}(\eta,t)=\mathcal{O}\left(\frac{1}{\eta^{2}}\right), \text{ for }\eta\rightarrow \infty.
\end{split}
\end{equation}

Comparing \eqref{shrodinger-tilde-Y-2} with \cite[(2.2)]{APC}, one can find that the two turning points in \eqref{shrodinger-tilde-Y-2} are not only coalesce at $\eta=0$ with the speed of $\mathcal{O}(\frac{1}{\sqrt{t}})$, but also they coalesce with a double pole $\eta=0$ as $t\rightarrow+\infty$. Therefore, the parabolic cylinder functions are not available here.
The form of (\ref{shrodinger-tilde-Y-2}) motivate us to consider the following model equation
\begin{equation}\label{shrodinger-equation-varphi-1}
\frac{d^{2}\varphi}{d\eta^2}=\left[\frac{t^{2}\eta^{2}}{4(\eta^{2}-\frac{1}{4})}+\frac{3}{4\eta^{2}}\right]\varphi.
\end{equation}
By careful analysis and according to the ideas in \cite{DTM}, we find that (\ref{shrodinger-equation-varphi-1}) is solvable. In fact, if we let
$$\varphi=\eta^{-\frac{1}{2}}\left(\eta^{2}-\frac{1}{4}\right)^{\frac{1}{2}}W(z)~\text{and }z=\frac{t}{2}\left(\eta^{2}-\frac{1}{4}\right)^{\frac{1}{2}},$$
then $W(z)$ satisfies the following modified Bessel equation
\begin{equation*}
\frac{d^{2}W}{dz^{2}}+\frac{1}{z}\frac{dW}{dz}-(1+\frac{1}{z^{2}})W=0
\end{equation*}
which has two independent solutions $K_{1}(z)$ and $I_{1}(z)$, cf. \cite[(10.25.1)]{OL}.

It is worth mentioning that, we can not use a pair of independent solutions of \eqref{shrodinger-equation-varphi-1} to asymptotically approximate $\tilde{\phi}$ uniformly everywhere on the Stokes lines, since $\tilde{g}(\eta,t)$ is not bounded as $\eta\rightarrow\frac{1}{2}$ (i.e. $\lambda\rightarrow 1$). Even so, the following lemma holds for $\eta\in \mathbb{C}\setminus [\frac{1}{2},+\infty)$.

\begin{lemma}\label{lemma-uniform-approximate-1}
There exists two constants $C_{1}$ and $C_{2}$ such that
\begin{equation}\label{uniform-asymptotic-bessel-1}
\tilde{\phi}=[C_{1}+o(1)]\varphi_{+}(\eta)+[C_{2}+o(1)]\varphi_{-}(\eta)
\end{equation}
uniformly for $\eta$ on two adjacent Stokes curves of (\ref{shrodinger-tilde-Y-2}) with $\eta\in \mathbb{C}\setminus [\frac{1}{2},+\infty)$ as $t\rightarrow\infty$, where
$$\varphi_{+}=\eta^{-\frac{1}{2}}(\eta^{2}-\frac{1}{4})^{\frac{1}{2}}K_{1}(\frac{t}{2}(\eta^{2}-\frac{1}{4})^{\frac{1}{2}}),$$
$$\varphi_{-}=\eta^{-\frac{1}{2}}(\eta^{2}-\frac{1}{4})^{\frac{1}{2}}I_{1}(\frac{t}{2}(\eta^{2}-\frac{1}{4})^{\frac{1}{2}})$$
are two linearly independent solutions of \eqref{shrodinger-equation-varphi-1}.
\end{lemma}

\begin{proof}
The proof of this lemma is similar to Lemma \ref{lemma-uniform-approximation-1} or \cite[Theorem 1 or 2]{APC} and it is also based on the parametrix variation method for the non-homogeneous ODEs. Precisely, we only need to show the following approximation
\begin{equation}\label{approximation-Integral-2}
\int_{\eta_{0}}^{\eta}\frac{\varphi_{+}(\eta)\varphi_{-}(s)-\varphi_{-}(\eta)\varphi_{+}(s)}{W(\varphi_{+},\varphi_{-})}\tilde{g}(s,t)\tilde{\phi}(s)ds=o(1)(\varphi_{+}(\eta)+\varphi_{-}(\eta)),~~t\rightarrow+\infty, \end{equation}
where the path of integration is taken along the Stokes curves and $W(\varphi_{+},\varphi_{-})$ is the Wronskian determinant. First, simple calculation yields $W(\varphi_{+},\varphi_{-})\equiv 1$ according to \cite[(10.28.2)]{OL}. In addition, since $\tilde{g}(\eta,t)=\mathcal{O}\left(\frac{1}{\eta^2}\right)$ as $\eta\rightarrow\infty$, then the integral in (\ref{approximation-Integral-2}) is really integrable.
Finally, by analysis we find that $\varphi_{\pm}=\mathcal{O}(1/\sqrt{t})$ as $t\rightarrow+\infty$ uniformly for all $\eta$ on the Stokes curves according to \cite[(10.3.1)(10.3.2)(10.40.2)(10.40.5)]{OL}.
This implies (\ref{approximation-Integral-2}) accordingly.
\end{proof}

\begin{remark}
By careful analysis, we find that (\ref{uniform-asymptotic-bessel-1}) is not true for $\eta\rightarrow \frac{1}{2}$. It implies that $\phi$ cannot be asymptotically approximated by the modified Bessel functions in the whole sector region $\arg\eta\in[\frac{-\pi}{2},\frac{\pi}{2}]$. This may be the reason why we cannot use Lemma \ref{lemma-uniform-approximate-1} to derive $s_{1}$, but should make another transformation (\ref{transform-Y-hat(Y)-1}).
\end{remark}

According to Lemma \ref{lemma-uniform-approximate-1}, $Y_{2}$, the second line of $\tilde{Y}$, i.e. $(Y_{21}, Y_{22})$, can be approximated by a linear combination of $\varphi_{+}$ and $\varphi_{-}$, and it is uniformly valid for all $\eta$ on the two adjacent Stokes curves which extend to $\infty$ with $\arg\eta=\frac{\pi}{2}$ and $\arg\eta=\frac{3\pi}{2}$ respectively. Then $s_{2}$ can be evaluated through $\tilde{Y}^{(3)}=\tilde{Y}^{(2)}S_{2}$.
Here and hereafter, we assume $\eta\gg t$ to ensure that $\frac{t}{2}\sqrt{\eta^{2}-\frac{1}{4}}\sim \frac{t\lambda}{2}-\frac{t}{4}$.


For $\eta\rightarrow\infty$ with $\arg\eta\sim \frac{\pi}{2}$, $\arg\left(\sqrt{\eta^{2}-\frac{1}{4}}\right)\sim\frac{\pi}{2}$. Using the uniform asymptotic behaviors of $K_{1}(z)$ and $I_{1}(z)$(see \cite[(10.40.2), (10.40.5)]{OL})
\begin{equation}\label{uniform-behavior-K-I-pi/2}
\begin{cases}
K_{1}(z)\sim (\frac{\pi}{2})^{\frac{1}{2}}z^{-\frac{1}{2}}e^{-z},&\arg z\in(-\frac{3\pi}{2},\frac{3\pi}{2});\\[0.2cm]
I_{1}(z)\sim (\frac{1}{2\pi})^{\frac{1}{2}}z^{-\frac{1}{2}}e^{z}-i(\frac{1}{2\pi})^{\frac{1}{2}}z^{-\frac{1}{2}}e^{-z},&\arg z\in(-\frac{\pi}{2},\frac{3\pi}{2}).
\end{cases}
\end{equation}
we get
\begin{equation}
\begin{cases}
\varphi_{+}\sim \pi^{\frac{1}{2}}t^{-\frac{1}{2}}e^{\frac{t}{4}}e^{-\frac{t\lambda}{2}},\\[0.2cm]
\varphi_{-}\sim \pi^{-\frac{1}{2}}t^{-\frac{1}{2}}e^{-\frac{t}{4}}e^{\frac{t\lambda}{2}}-i\pi^{-\frac{1}{2}}t^{-\frac{1}{2}}e^{\frac{t}{4}}e^{-\frac{t\lambda}{2}}.
\end{cases}
\end{equation}
Hence

\begin{equation}\label{uniform-asymptotic-behavior-1}
\begin{cases}
(\tilde{C})^{\frac{1}{2}}\varphi_{+}\sim i\pi^{\frac{1}{2}}e^{\frac{t}{4}}e^{-\frac{t\lambda}{2}}=d_{1}e^{-\frac{t\lambda}{2}},\\[0.2cm]
(\tilde{C})^{\frac{1}{2}}\varphi_{-}\sim i\pi^{-\frac{1}{2}}e^{-\frac{t}{4}}e^{\frac{t\lambda}{2}}+\pi^{-\frac{1}{2}}e^{\frac{t}{4}}e^{-\frac{t\lambda}{2}}=d_{2}e^{\frac{t\lambda}{2}}+d_{3}e^{-\frac{t\lambda}{2}}.
\end{cases}
\end{equation}
In virtue of the asymptotic behavior of the canonical solution $Y(\lambda)$ in \eqref{1.8} and the transformation \eqref{transform-Y-tilde(Y)-1}, we get
$$\tilde{Y}_{21}(\lambda)\sim-u^{-\frac{1}{2}}e^{\frac{t\lambda}{2}},~~\tilde{Y}_{22}(\lambda)\sim u^{\frac{1}{2}}e^{-\frac{t\lambda}{2}}\quad \text{ as }\lambda\rightarrow\infty, \lambda\in\Omega^{(2)}.$$
Comparing with \eqref{uniform-asymptotic-behavior-1}, one can easily obtain
\begin{equation}\label{uniform-tilde-Y-1}
(\tilde{Y}_{21},\tilde{Y}_{22})\sim (\tilde{C})^{\frac{1}{2}}(\varphi_{+},\varphi_{-})\left(\begin{matrix}\frac{d_{3}u^{-\frac{1}{2}}}{d_{1}d_{2}}&\frac{u^{\frac{1}{2}}}{d_{1}}\\[0.2cm] -\frac{u^{-\frac{1}{2}}}{d_{2}}&0\end{matrix}\right) \text{ as } \lambda\rightarrow\infty \text{ in } \Omega^{(2)}.
\end{equation}

Now for the case $\arg\eta=\arg(\sqrt{\eta^{2}-\frac{1}{4}})\sim\frac{3\pi}{2}$. Here we need the following analytic continuation formula for the modified Bessel functions $K_{1}(z)$ and $I_{1}(z)$ (see \cite[(10.34.5), (10.34.6)]{OL})
\begin{equation}\label{analytic-continuation-for-K-I}
\begin{cases}
K_{1}(z)=-K_{1}(ze^{-2\pi i})-2K_{1}(ze^{-\pi});\\[0.2cm]
I_{1}(z)=\frac{1}{\pi i}\left(K_{1}(ze^{-\pi i})+K_{1}(z)\right).
\end{cases}
\end{equation}
Substituting \eqref{uniform-behavior-K-I-pi/2} into \eqref{analytic-continuation-for-K-I}, we can obtain the uniform asymptotic behavior of $K_{1}(z)$ and $I_{1}(z)$ for $\arg z\sim \frac{3\pi}{2}$
\begin{equation}\label{uniform-behavior-K-I-3pi/2}
\begin{cases}
K_{1}(z)\sim (\frac{\pi}{2})^{\frac{1}{2}}z^{-\frac{1}{2}}e^{-z}-2i(\frac{\pi}{2})^{\frac{1}{2}}z^{-\frac{1}{2}}e^{z},&\arg z\in(-\frac{3\pi}{2},\frac{3\pi}{2});\\[0.2cm]
I_{1}(z)\sim -(\frac{1}{2\pi})^{\frac{1}{2}}z^{-\frac{1}{2}}e^{z}-i(\frac{1}{2\pi})^{\frac{1}{2}}z^{-\frac{1}{2}}e^{-z},&\arg z\in(\frac{\pi}{2},\frac{5\pi}{2}).
\end{cases}
\end{equation}
By suitable modification to the deriving of \eqref{uniform-tilde-Y-1}, we obtain that
\begin{equation}\label{3.15}
(\tilde{Y}_{21},\tilde{Y}_{22})\sim (\tilde{C})^{\frac{1}{2}}(\varphi_{+},\varphi_{-})\left(\begin{matrix}\frac{-e_{4}u^{-\frac{1}{2}}}{M}&\frac{e_{3}u^{\frac{1}{2}}}{M}\\[0.2cm]
\frac{e_{1}u^{-\frac{1}{2}}}{M}&\frac{e_{2}u^{\frac{1}{2}}}{M}\end{matrix}\right) \text{ as } \lambda\rightarrow\infty \text{ in } \Omega^{(3)},
\end{equation}
where $e_{1}=d_{1}, e_{2}=-2\pi id_{2}, e_{3}=d_{2}, e_{4}=d_{3}=\frac{-i}{\pi}d_{1}$ and $M=e_{1}e_{3}+e_{2}e_{4}=-d_{1}d_{2}$.

By using of the definition of $S_2$ in \eqref{1.9}, it follows from \eqref{uniform-tilde-Y-1} and \eqref{3.15} that
\begin{equation*}
S_{2}\sim\left(\begin{matrix}\frac{d_{3}u^{-\frac{1}{2}}}{d_{1}d_{2}}&\frac{u^{\frac{1}{2}}}{d_{1}}\\[0.2cm] -\frac{u^{-\frac{1}{2}}}{d_{2}}&0\end{matrix}\right)^{-1}\left(\begin{matrix}\frac{-e_{4}u^{-\frac{1}{2}}}{M}&\frac{e_{3}u^{\frac{1}{2}}}{M}\\[0.2cm]
\frac{e_{1}u^{-\frac{1}{2}}}{M}&\frac{e_{2}u^{\frac{1}{2}}}{M}\end{matrix}\right)=\left(\begin{matrix}1&\frac{-2\pi id_{2}u}{d_{1}}\\[0.2cm]
0&1\end{matrix}\right),
\end{equation*}
which gives
\begin{equation}\label{t-infty-s2}
s_{2}=-2i\hat{u}.
\end{equation}
Here, use has been made of the fact that $s_{2}$ is independent of $t$.

Using the similar method as in the computation of $s_{2}$, we can carry out $s_{1}$.
If we replace the transformation \eqref{transform-Y-tilde(Y)-1} by
%
\begin{equation}{\label{transform-Y-hat(Y)-1}}
\hat{Y}(\lambda)=\left(\begin{matrix}1&0\\1&1\end{matrix}\right)(-uy)^{-\frac{\sigma_{3}}{2}}Y(\lambda),
\end{equation}
then the Shr\"{o}dinger equation \eqref{shrodinger-tilde-Y-2} becomes
\begin{equation}\label{shrodinger-equation-hat-Y-2}
\begin{split}
\frac{d^{2}\hat{\phi}}{d\eta^2}=\left[\frac{t^{2}\eta^{2}}{4(\eta^{2}-\frac{1}{4})}+\frac{3}{4\eta^{2}}+\hat{g}(\eta,t)\right]\hat{\phi}=\hat{F}(\eta,t)\hat{\phi}.
\end{split}
\end{equation}
where $\hat{g}(\eta,t)$ is bounded  for $\eta$ away from both $0$ and $-\frac{1}{2}$, and
\begin{equation*}
\begin{split}
&\hat{g}(\eta,t)=\mathcal{O}\left(\frac{1}{\eta}\right)  \text{ as }\eta\rightarrow 0;~~~~\hat{g}(\eta,t)=\mathcal{O}\left(\frac{1}{(\eta+\frac{1}{2})^{2}}\right) \text{ as }\eta\rightarrow -\frac{1}{2};\\
&\hat{g}(\eta,t)=\mathcal{O}\left(\frac{1}{\eta^{2}}\right) \text{ as }\eta\rightarrow \infty.
\end{split}
\end{equation*}
Hence we obtain a similar result to Lemma \ref{lemma-uniform-approximate-1} as follows:
\begin{lemma}\label{lemma-uniform-approximate-2}
There exists two constants $C_{1}$ and $C_{2}$ such that
\begin{equation}\label{uniform-approximation-2}
\hat{\phi}=[C_{1}+o(1)]\varphi_{+}(\eta)+[C_{2}+o(1)]\varphi_{-}(\eta)
\end{equation}
uniformly for $\eta$ on Stokes curves with $\eta\in \mathbb{C}\setminus (-\infty,-\frac{1}{2}]$ as $t\rightarrow\infty$, where
$\varphi_{\pm}$ are two linearly independent solutions of \eqref{shrodinger-equation-varphi-1} defined in Lemma \ref{lemma-uniform-approximate-1}.
\end{lemma}


With the help of the Lemma \ref{lemma-uniform-approximate-2}, we need only repeat the process of the derivation of $s_{2}$ above.
According to \cite[(10.40.2), (10.40.5)]{OL}, the uniform asymptotic behaviors of $K_{1}(z)$ and $I_{1}(z)$ are
\begin{equation}\label{uniform-behavior-K-I--pi/2}
\begin{cases}
K_{1}(z)\sim (\frac{\pi}{2})^{\frac{1}{2}}z^{-\frac{1}{2}}e^{-z},&\arg z\in(-\frac{3\pi}{2},\frac{3\pi}{2});\\[0.2cm]
I_{1}(z)\sim (\frac{1}{2\pi})^{\frac{1}{2}}z^{-\frac{1}{2}}e^{z}+i(\frac{1}{2\pi})^{\frac{1}{2}}z^{-\frac{1}{2}}e^{-z},&\arg z\in(-\frac{3\pi}{2},\frac{\pi}{2}),\\[0.2cm]
I_{1}(z)\sim (\frac{1}{2\pi})^{\frac{1}{2}}z^{-\frac{1}{2}}e^{z}-i(\frac{1}{2\pi})^{\frac{1}{2}}z^{-\frac{1}{2}}e^{-z},&\arg z\in(-\frac{\pi}{2},\frac{3\pi}{2}).
\end{cases}
\end{equation}
Hence we can easily get
\begin{equation}\label{uniform-asymptotic-behavior-3}
\begin{cases}
(\hat{C})^{\frac{1}{2}}\varphi_{+}\sim \pi^{\frac{1}{2}}e^{\frac{t}{4}}e^{-\frac{t\lambda}{2}}=\frac{{d}_{1}}{i}e^{-\frac{t\lambda}{2}},\arg z\in(-\frac{3\pi}{2},\frac{3\pi}{2});\\[0.2cm]
(\hat{C})^{\frac{1}{2}}\varphi_{-}\sim \pi^{-\frac{1}{2}}e^{-\frac{t}{4}}e^{\frac{t\lambda}{2}}+i\pi^{-\frac{1}{2}}e^{\frac{t}{4}}e^{-\frac{t\lambda}{2}}=\frac{{d}_{2}}{i}e^{\frac{t\lambda}{2}}-\frac{{d}_{3}}{i}e^{-\frac{t\lambda}{2}},\arg z\in(-\frac{3\pi}{2},\frac{\pi}{2});\\[0.2cm]
(\hat{C})^{\frac{1}{2}}\varphi_{-}\sim \pi^{-\frac{1}{2}}e^{-\frac{t}{4}}e^{\frac{t\lambda}{2}}-i\pi^{-\frac{1}{2}}e^{\frac{t}{4}}e^{-\frac{t\lambda}{2}}=\frac{{d}_{2}}{i}e^{\frac{t\lambda}{2}}+\frac{{d}_{3}}{i}e^{-\frac{t\lambda}{2}},\arg z\in(-\frac{\pi}{2},\frac{3\pi}{2});
\end{cases}
\end{equation}
where $d_{j},j=1,2,3$ are defined in (\ref{uniform-asymptotic-behavior-1}) and $\hat{C}=t+\frac{2v-vy-v/y}{\lambda}$. Noting that the large-$\lambda$ asymptotic behavior of $\hat{Y}(\lambda)$ has also changed to be
\begin{equation}\label{uniform-hat-Y--pi/2}
\hat{Y}_{21}(\lambda)\sim(-uy)^{-\frac{1}{2}}e^{\frac{t\lambda}{2}},~~\hat{Y}_{22}(\lambda)\sim (-uy)^{\frac{1}{2}}e^{-\frac{t\lambda}{2}},~~\lambda\rightarrow\infty, \lambda\in\Omega^{(k)},
\end{equation}
we get
\begin{equation}{\label{Y-t-infty--pi/2}}
(\hat{Y}_{21},\hat{Y}_{22})\sim(\hat{C})^{\frac{1}{2}}(\varphi_{+},\varphi_{-})\left(\begin{matrix}\frac{id_{3}(-uy)^{-\frac{1}{2}}}{d_{1}d_{2}}&\frac{i(-uy)^{\frac{1}{2}}}{d_{1}}\\[0.2cm]
\frac{i(-uy)^{-\frac{1}{2}}}{d_{2}}&0\end{matrix}\right),~~\lambda\rightarrow\infty, \lambda\in \Omega^{(1)}.
\end{equation}
and
\begin{equation}{\label{Y-t-infty-pi/2}}
(\hat{Y}_{21},\hat{Y}_{22})\sim(\hat{C})^{\frac{1}{2}}(\varphi_{+},\varphi_{-})\left(\begin{matrix}-\frac{id_{3}(-uy)^{-\frac{1}{2}}}{d_{1}d_{2}}&\frac{i(-uy)^{\frac{1}{2}}}{d_{1}}\\[0.2cm]
\frac{i(-uy)^{-\frac{1}{2}}}{d_{2}}&0\end{matrix}\right),~~\lambda\rightarrow\infty, \lambda\in \Omega^{(2)}.
\end{equation}
Noting that in (\ref{Y-t-infty--pi/2}), $(Y_{21},Y_{22})$ represent the second line of $Y^{(1)}(\lambda)$, while in (\ref{Y-t-infty-pi/2}), $(Y_{21},Y_{22})$ is the second line of $Y^{(2)}(\lambda)$, according to the definition of $S_1$ in (\ref{1.9}), it follows that

\begin{equation*}
S_{1}\sim\left(\begin{matrix}\frac{id_{3}(-uy)^{-\frac{1}{2}}}{d_{1}d_{2}}&\frac{i(-uy)^{\frac{1}{2}}}{d_{1}}\\[0.2cm]
\frac{i(-uy)^{-\frac{1}{2}}}{d_{2}}&0\end{matrix}\right)^{-1}\left(\begin{matrix}-\frac{id_{3}(-uy)^{-\frac{1}{2}}}{d_{1}d_{2}}&\frac{i(-uy)^{\frac{1}{2}}}{d_{1}}\\[0.2cm]
\frac{i(-uy)^{-\frac{1}{2}}}{d_{2}}&0\end{matrix}\right)=\left(\begin{matrix}1&0\\[0.2cm] \frac{2d_{3}}{d_{2}uy}&1\end{matrix}\right),
\end{equation*}
which gives
\begin{equation}\label{t-infty-s1}
s_{1}=\frac{2i}{\hat{u}}.
\end{equation}

Finally, combining \eqref{t-infty-s1} with \eqref{t-infty-s2} and \eqref{t-0-s1} with \eqref{t-0-s2}, we can immediately have
$s_{1}s_{2}=4=-4\sin^{2}{\frac{\pi\sigma}{2}}$, which gives \eqref{connection formula} accordingly.
 \medskip

\section*{Acknowledgements}
The authors are grateful to Prof. Yu-Qiu Zhao for valuable discussions and suggestions.
This work was supported in part by the National Natural Science Foundation of China under grant number 11201070, and the GuangDong Natural Science Foundation under grant number Yq2013161 and 2014A030313176.

\end{document}